\documentclass[11pt]{amsart}%
\usepackage{palatino, mathpazo}
\usepackage[mathscr]{eucal}
\usepackage{amssymb}
\usepackage{graphicx}
\usepackage{color}
\usepackage{amsfonts}
\usepackage{amsmath}
\usepackage{amssymb,latexsym}%
\providecommand{\U}[1]{\protect \rule{.1in}{.1in}}
\newtheorem{theorem}{Theorem}[section]

\newtheorem{corollary}[theorem]{Corollary}

\newtheorem{definition}[theorem]{Definition}

\newtheorem{lemma}[theorem]{Lemma}

\newtheorem{proposition}[theorem]{Proposition}
\newtheorem{remark}[theorem]{Remark}

\numberwithin{equation}{section}
\begin{document}
\title[Painlev\'{e} VI equation]{Unitary monodromy implies the smoothness along the real axis for some
Painlev\'{e} VI equation, I}
\author{Zhijie Chen}
\address{Department of Mathematical Sciences, Yau Mathematical Sciences Center,
Tsinghua University, Beijing, 100084, China }
\email{zjchen@math.tsinghua.edu.cn}
\author{Ting-Jung Kuo}
\address{Taida Institute for Mathematical Sciences (TIMS), National Taiwan University,
Taipei 10617, Taiwan }
\email{tjkuo1215@gmail.com}
\author{Chang-Shou Lin}
\address{Taida Institute for Mathematical Sciences (TIMS), Center for Advanced Study in
Theoretical Sciences (CASTS), National Taiwan University, Taipei 10617, Taiwan }
\email{cslin@math.ntu.edu.tw}

\begin{abstract}
In this paper, we study the Painlev\'{e} VI equation with parameter $(\frac
{9}{8},\frac{-1}{8},\frac{1}{8},\frac{3}{8})$. We prove

(i) An explicit formula to count the number of poles of an algebraic solution
with the monodromy group of the associated linear ODE being $D_{N}$, where $D_{N}$ is the dihedral group of order
$2N$.

(ii) There are only four solutions without poles in $\mathbb{C}\backslash
\{0,1\}  $.

(iii) If the monodromy group of the associated linear ODE of a solution
$\lambda (t)$ is unitary, then $\lambda(t)  $ has
no poles in $\mathbb{R}\backslash \{0,1\}  $.

\end{abstract}
\maketitle

\section{Introduction}

In literature, Painlev\'{e} VI equation with four free parameters
$(\alpha,\beta,\gamma,\delta)$ (PVI$(\alpha,\beta,\gamma,\delta)$) can be
written as%
\begin{align}
\frac{d^{2}\lambda}{dt^{2}}=  &  \frac{1}{2}\left(  \frac{1}{\lambda}+\frac
{1}{\lambda-1}+\frac{1}{\lambda-t}\right)  \left(  \frac{d\lambda}{dt}\right)
^{2}-\left(  \frac{1}{t}+\frac{1}{t-1}+\frac{1}{\lambda-t}\right)
\frac{d\lambda}{dt}\nonumber \\
&  +\frac{\lambda(\lambda-1)(\lambda-t)}{t^{2}(t-1)^{2}}\left[  \alpha
+\beta \frac{t}{\lambda^{2}}+\gamma \frac{t-1}{(\lambda-1)^{2}}+\delta
\frac{t(t-1)}{(\lambda-t)^{2}}\right]  . \label{46}%
\end{align}

There are two fundamental facts about PVI$(\alpha,\beta,\gamma,\delta)$
(\ref{46}). The first one is the \emph{Painlev\'{e} property} which says that
both the branch points and essential singularities of any solution
$\lambda(t)  $ of (\ref{46}) are independent of any particular
solution and consist of $0,1,\infty$ only. Thus $\lambda(t) $ is a
multi-valued meromorphic function in $\mathbb{C}\backslash \{0,1\}$; naturally
it can be lifted to the universal covering $\mathbb{H}=\left \{  \tau \text{
}|\text{ }\operatorname{Im}\tau>0\right \}  $ of $\mathbb{C}\backslash \{0,1\}$
through the transformation:%
\begin{equation}
t(\tau)=\frac{e_{3}(\tau)-e_{1}(\tau)}{e_{2}(\tau)-e_{1}(\tau)}\text{ \  \ and
\  \ }\lambda(t)=\frac{\wp(p(\tau)|\tau)-e_{1}(\tau)}{e_{2}(\tau)-e_{1}(\tau)}.
\label{tr}%
\end{equation}
where $\wp( z|\tau) $ is the Weierstrass elliptic function with periods $1$
and $\tau$.

Throughout the paper, we use the notations $\omega_{0}=0$, $\omega_{1}=1$,
$\omega_{2}=\tau$, $\omega_{3}=1+\tau,$ $e_{k}=e_{k}(\tau)\doteqdot \wp(
\frac{\omega_{k}}{2}|\tau) ,$ $k=1,2,3$ and $\Lambda_{\tau}=\mathbb{Z+Z}\tau$,
where $\tau \in \mathbb{H}$. Define $E_{\tau}\doteqdot \mathbb{C}/\Lambda_{\tau}$
to be a flat torus in the plane and $E_{\tau}[2]\doteqdot \{ \frac{\omega_{i}%
}{2}$ $|$ $0\leq i\leq3\}+\Lambda_{\tau}$ to be the set consisting of the
lattice points and $2$-torsion points in $E_{\tau}$.

By the transformation (\ref{tr}), $p(\tau)$ satisfies the following
\emph{elliptic form} of PVI%
\begin{equation}
\frac{d^{2}p(\tau)}{d\tau^{2}}=\frac{-1}{4\pi^{2}}\sum_{k=0}^{3}\alpha_{k}%
\wp^{\prime}\left(  \left.  p(\tau)+\frac{\omega_{k}}{2}\right \vert
\tau \right)  , \label{124}%
\end{equation}
where $\wp^{\prime}(z|\tau)=\frac{d}{dz}\wp(z|\tau)$ and
\begin{equation}
\left(  \alpha_{0},\alpha_{1},\alpha_{2},\alpha_{3}\right)  =\left(
\alpha,-\beta,\gamma,\tfrac{1}{2}-\delta \right)  . \label{para}%
\end{equation}
See \cite{Babich-Bordag, Y.Manin} for a proof. As a solution to (\ref{124}),
$p( \tau) $ is considered as a continuous function from $\mathbb{H}$ to
$\mathbb{C}$ and holomorphic except at those $\tau$'s such that $p(\tau) $ $\in E_{\tau}[2]$. Although $p(
\tau) $ has branch points at $p( \tau) \in E_{\tau}[2]\setminus \Lambda
_{\tau}$, by the Painlev\'{e} property, $\wp(p(\tau)|\tau)$ is a
\emph{single-valued} \emph{meromorphic }function in $\mathbb{H}$ and its poles
are exactly equal to those $\tau$'s such that $p(\tau) \in \Lambda_{\tau}$.

Another important feature about (\ref{46}) is that for any solution
$\lambda(t)$, there associates with a second order Fuchsian ODE
defined on $\mathbb{CP}^{1}$, whose regular singular points are exactly
$\{0,1,t,\lambda(t) ,\infty \}$ with $\lambda(t)$ being an apparent singularity
(cf. \cite{GP}) such that the monodromy representation is invariant under the
deformation of $t$. See \cite{Jimbo-Miwa} for a more general theory. By using the transformation (\ref{tr}), this associated linear ODE on $\mathbb{CP}^1$
could be transformed into a new ODE defined on $E_{\tau}$, a generalized
Lam\'{e} equation (GLE) which can be written as follows:%
\begin{align}
y^{\prime \prime}(z)=  &  \Big[\sum \nolimits_{k=0}^{3}n_{k}(n_{k}+1)\wp \left(
z+\frac{\omega_{k}}{2}\right)  +\tfrac{3}{4}(\wp(z+p)\label{505}\\
&  +\wp(z-p))+A(\zeta(z+p)-\zeta(z-p))+B\Big]y(z),\nonumber
\end{align}
where $A,B$ are complex numbers and the parameters are related by%
\begin{equation}
\alpha_{k}=\tfrac{1}{2}\left(  n_{k}+\tfrac{1}{2}\right)  ^{2}\text{ \ for
\ }k\in \{0,1,2,3\}. \label{505-1}%
\end{equation}
In (\ref{505}), $\zeta(z)$ $=\zeta( z|\tau)\doteqdot-\int^{z}\wp(\xi|\tau
)d\xi$ is the Weierstrass zeta function. The function $\zeta(z)$ is an odd but
not elliptic function. Indeed, $\zeta( z) $ satisfies
\begin{equation}
\zeta( z+\omega_{k}|\tau) =\zeta( z|\tau) +\eta_{k}( \tau) ,\quad k=1,2,3,
\label{40-2}%
\end{equation}
where $\eta_{k}(\tau) $ are called the quasi-periods of $\zeta( z) $. The GLE
(\ref{505}) has regular singularities at $E_{\tau}[ 2]\cup(\{ \pm p( \tau)
\}+\Lambda_{\tau}) $. Similar to $\lambda(t)$, $p( \tau) $ is always an
\emph{apparent} singularity. The monodromy representation of (\ref{505}) is
invariant under the deformation of $\tau$ if and only if $p(\tau) $ is a
solution of (\ref{124}). This fact might be indirectly proved by the
transformation (\ref{tr}). In \cite{Chen-Kuo-Lin}, this fact and the
associated Hamiltonian system has been directly derived.

Naturally, this isomonodromic feature proposes the following question:
\emph{How the monodromy group of the associated linear ODE effects the
solution }$\lambda(t) $\emph{ or }$p( \tau) $\emph{? }

For certain parameters $\alpha_{k}$, the monodromy representation of the GLE
(\ref{505}) associated with $p(\tau)$ is easier to compute than that of the
associated ODE with $\lambda(t)$. For example, if all $n_{k}$ $\in \mathbb{Z}$,
then the monodromy representation of (\ref{505}) is reduced to a homomorphism
from\ $\pi_{1}(E_{\tau})$ to $SL(2,\mathbb{C})$. This fact immediately implies that the
monodromy group is always \emph{abelian}, a significant reduction. See
\cite{Chen-Kuo-Lin2}. Thus, there always exists a common eigenfunction to all
the monodromy matrices of the GLE (\ref{505}).

\begin{definition}
Let the parameter $(\alpha,\beta,\gamma,\delta)$ be given by (\ref{para}) and
(\ref{505-1}) with $n_{k}\in \mathbb{Z}$ for all $k$.

(i) A solution $\lambda(t)$ of (\ref{46}) or $p(\tau)$ of the elliptic form
(\ref{124}) is called completely reducible if the monodromy representation of
the associated GLE (\ref{505}) is completely reducible (i.e. all monodromy
matrices can be diagonalized simultaneously); otherwise, $\lambda(t)$ or $p(
\tau) $ is called non-completely reducible.

(ii) A solution $\lambda(t)$ or $p(\tau)$ is called an unitary solution if up to a common conjugation, the
monodromy group of the associated GLE (\ref{505}) is contained in the unitary
group $SU(2)$.
\end{definition}

Obviously, any unitary solution is completely reducible. We remark that an
unitary solution is related to the existence of conformal metric with constant
curvature $+1$ (those metrics are with conic singularities in general). See
\cite{EGH}.

Now we come back to the problem we are concerned above. For PVI$(\frac{1}%
{8},\frac{-1}{8},$ $\frac{1}{8},\frac{3}{8})$ (i.e. $n_{k}=0$ in
(\ref{505-1}), $\forall k=0,1,2,3$, the well-known case studied by Hitchin
\cite{Hit1}), among other things, we proved in \cite{CKLW} the following
theorem: \medskip

\noindent \textbf{Theorem A.} (\cite{CKLW})

\begin{itemize}
\item[(i)] \emph{PVI}$(\frac{1}{8},\frac{-1}{8},\frac{1}{8},\frac{3}{8}%
)$\emph{ has exactly three solutions which are completely reducible and
satisfy} $\lambda(t)\not \in \{0,1,t,\infty \}$ \emph{for any} $t\in
\mathbb{C}\backslash \{0,1\}$.

\item[(ii)] \emph{Any unitary solution }$\lambda( t) $\emph{ of PVI}$(\frac
{1}{8},\frac{-1}{8},\frac{1}{8},\frac{3}{8})$\emph{ has no poles in
}$\mathbb{R}\backslash \{0,1\}$\emph{.}\medskip
\end{itemize}

By using the Okamoto transformation \cite{Okamoto1}, it is well known that any
solution of PVI$(\alpha,\beta,\gamma,\delta)$ with parameters given by
(\ref{505-1}) with $n_{k}\in \mathbb{Z}$ could be obtained from solution
$\lambda(t)$ of PVI$(\frac{1}{8},\frac{-1}{8},\frac{1}{8},\frac{3}{8})$ (i.e.
$n_{k}=0$ for all $k$). However, the Okamoto transformation is a rational map
of $\lambda(t)$ and $\lambda^{\prime}(t)$. So obviously, the smoothness of
solutions can not be preserved by the Okamoto transformation. See \cite{CKLW}
for discussions of this issue.

In this article, we extend Theorem A (ii) to the case $\left(  n_{0}%
,n_{1},n_{2},n_{3}\right)  =\left(  1,0,0,0\right)  $.

\begin{theorem}
\label{theorem1-7 copy(1)}Let $\lambda(t)$ be any unitary solution of
PVI$(\frac{9}{8},\frac{-1}{8},\frac{1}{8},\frac{3}{8})$. Then $\lambda(t)$ has
no poles in $\mathbb{R}\backslash \{0,1\}$, namely $\lambda(t)$ is holomorphic
in $\mathbb{R}\backslash \{0,1\}$.
\end{theorem}

As discussed above, Theorem \ref{theorem1-7 copy(1)} can not be obtained from
Theorem A by the Okamoto transformation. Our proof is based on the
generalization of the famous Hitchin theorem for PVI$(\frac{1}{8},\frac{-1}%
{8},\frac{1}{8},\frac{3}{8})$ to PVI$(\frac{9}{8},\frac{-1}{8},\frac{1}%
{8},\frac{3}{8})$.\medskip

\noindent \textbf{Theorem B. }(\cite{Chen-Kuo-Lin2}) $p(\tau) $ \emph{is a
completely reducible solution to the elliptic form of PVI}$(\frac{9}{8}%
,\frac{-1}{8},\frac{1}{8},\frac{3}{8})$\emph{ if and only if there exists a
fixed pair }$(r,s)\in \mathbb{C}^{2}\backslash \frac{1}{2}\mathbb{Z}^{2}$\emph{
such that}%
\begin{equation}
\wp(p(\tau)|\tau)=\wp(\alpha)+\frac{3\wp^{\prime}(\alpha)Z_{r,s}^{2}+\left(
12\wp^{2}(\alpha)-g_{2}\right)  Z_{r,s}+3\wp(\alpha)\wp^{\prime}(\alpha
)}{2\left(  Z_{r,s}^{3}-3\wp(\alpha)Z_{r,s}-\wp^{\prime}(\alpha)\right)  },
\label{625}%
\end{equation}
\emph{where }$\alpha=\alpha(\tau)\doteqdot r+s\tau$,%
\begin{equation}
Z_{r,s}=Z_{r,s}(\tau)\doteqdot \zeta(r+s\tau|\tau)-r\eta_{1}(\tau)-s\eta
_{2}(\tau), \label{Hecke}%
\end{equation}
\emph{ and} $g_{2}=g_{2}(\tau)$ \emph{is the coefficient of}
\begin{equation}
\label{II-33}\wp^{\prime2}=4\wp(z|\tau)^{3}-g_{2}(\tau)\wp(z|\tau)-g_{3}%
(\tau).
\end{equation}

The formula (\ref{625}) was first obtained by Takemura \cite{Takemura} and
also obtained in \cite{Chen-Kuo-Lin2} by a different argument. We will see
that (\ref{625}) plays a fundamental role for studying the poles of
$\lambda(t)$.

\begin{remark}
Let $p( \tau) $ be the solution given by (\ref{625}) and (\ref{505}) be the
associated ODE with this solution. Then it is proved in \cite{Chen-Kuo-Lin2}
that there is a pair of independent solutions $y_{i}( z|\tau) $, $i=1,2$, such
that their analytic continuation along path $\ell_{i}$ satisfy%
\begin{align*}
\ell_{1}^{\ast}\left(  y_{1},y_{2}\right)   &  =\left(  y_{1},y_{2}\right)
\begin{pmatrix}
e^{2\pi is} & 0\\
0 & e^{-2\pi is}%
\end{pmatrix}
,\\
\ell_{2}^{\ast}\left(  y_{1},y_{2}\right)   &  =\left(  y_{1},y_{2}\right)
\begin{pmatrix}
e^{-2\pi ir} & 0\\
0 & e^{2\pi ir}%
\end{pmatrix}
,
\end{align*}
where $\left(  r,s\right)  $ is the pair in (\ref{625}) and $\ell_{i},i=1,2$,
are two fundamental cycles on $E_{\tau}$ with the base point $q_{0}%
\not \in E_{\tau}[ 2]\cup(\{ \pm p( \tau) \}+\Lambda_{\tau}) $. In this paper,
the solution given by (\ref{625}) is denoted by $p_{r,s}( \tau) $ and the
corresponding $\lambda( t) $ through (\ref{tr}) by $\lambda_{r,s}( t) $.
\end{remark}

Obviously, $\lambda_{r,s}( t) $ (or $p_{r,s}(\tau) $) is unitary if and only
if $( r,s) \in \mathbb{R}^{2}\backslash \frac{1}{2}\mathbb{Z}^{2}$. Furthermore,
$\lambda_{r,s}(t) $ is an algebraic solution if and only if $( r,s) $ is a
$N$-torsion point for some $N\in \mathbb{N}_{\geq3}$, that is $( r,s) \in
Q_{N}$ where%
\begin{equation}
Q_{N}\doteqdot \left \{  \left.  \left(  \tfrac{k_{1}}{N},\tfrac{k_{2}}%
{N}\right)  \right \vert \gcd(k_{1},k_{2},N)=1,\text{ }0\leq k_{1},k_{2}\leq
N-1\right \}  . \label{q-n}%
\end{equation}
For Painlev\'{e} VI equation, all the algebraic solutions have been classified
through the Okamoto transformation. For example, the monodromy group of the
associated linear ODE on $\mathbb{CP}^1$ (not the GLE on $E_\tau$!) of an algebraic solution of PVI$(\alpha,\beta
,\gamma,\delta)$ with parameters given by (\ref{para}) and (\ref{505-1}) with
$n_{k}\in \mathbb{Z}$ is always a dihedral group $D_{N}$ of order $2N$ for some
$N\in \mathbb{N}_{\geq3}$. See
\cite{Boalch,Dubrovin-Mazzocco,Lisovyy-Tykhyy,Mazzocco}. From the
classification of Theorem B, such a $\lambda( t) $ must be of the form
$\lambda_{r,s}( t) $ with $( r,s) \in Q_{N}, N\geq3$. Note that the problem
concerning the distribution of poles of Painlev\'{e} VI solution has been
addressed in \cite{Brezhnev,Guzzetti-1,Shimomura}.

Let $\phi(N)$ be the Euler function defined by%
\begin{equation}
\phi(N):=\# \{k\in \mathbb{Z}|\gcd(k,N)=1,0\leq k<N\}.\label{Euler}%
\end{equation}
In the following, we shall apply Theorem B to obtain the formula to count the
number of poles of algebraic solutions of PVI$(\frac{9}{8},\frac{-1}{8}%
,\frac{1}{8},\frac{3}{8})$.

\newpage

\begin{theorem}
\label{thm number}
\end{theorem}

\begin{itemize}
\item[(i)] \emph{If }$N$\emph{ is odd, then there is only one algebraic
solution whose monodromy group\footnote{In this paper, for convenience, when
we say the monodromy group of an algebraic solution, we always mean the one of
the associated linear ODE on $\mathbb{CP}^{1}$, but not the one of this algebraic solution as a
multi-valued function.} is the dihedral group }$D_{N}$\emph{. This algebraic
solution has exactly 
\[
\frac{3|Q_{N}|}{4}-3\phi( N)
\]
poles in} $\mathbb{C}\backslash \{0,1\}$.

\item[(ii)] \emph{If }$N$\emph{ is even, there are exactly three algebraic
solutions and each of them has }%
\[
\frac{|Q_{N}|}{4}-\left(  \phi( N) +\phi \left(  \frac{N}{2}\right)  \right)
\]
\emph{poles in }$\mathbb{C}\backslash \{0,1\}$\emph{.}
\end{itemize}

\begin{remark}
For any solution $\lambda( t) $, there might be two different branches to have
poles at the same $t_{0}$. In this case, we count the number of poles at
$t_{0}$ by $2$. It is our conjecture that different branches should not have
common poles. See the discussion before the proof of Theorem \ref{thm number}
in \S 3. This conjecture was proved for PVI$(\frac{1}{8},\frac{-1}{8},\frac
{1}{8},\frac{3}{8})$ in \cite{CKLW}.
\end{remark}

Finally, we extend Theorem A (i) to PVI$(\frac{9}{8},\frac{-1}{8},\frac{1}%
{8},\frac{3}{8})$.

\begin{theorem}
\label{main-thm copy(2)}Among all the completely reducible solutions
$\lambda_{r,s}(t)$ to PVI$(\frac{9}{8},$ $\frac{-1}{8},\frac{1}{8},\frac{3}%
{8})$, there are exactly four solutions which have no poles in $\mathbb{C}%
\backslash \{0,1\}$. They are $\lambda_{\frac{1}{3},0}(t), \lambda_{\frac{1}%
{4},0}(t), \lambda_{\frac{1}{4},\frac{1}{4}}(t)$ and $\lambda_{0,\frac{1}{4}%
}(t)$.
\end{theorem}

\begin{remark}
We proved in \cite{Chen-Kuo-Lin2} that $\lambda_{\frac{1}{4},0}(t)=-\frac{1}{3}t^{\frac{1}{2}}$. Together with (\ref{4solution}) in Section 3, we immediately obtain
\[\lambda_{0,\frac{1}{4}}(t)=1+\frac{1}{3}(1-t)^{\frac{1}{2}},\quad \lambda_{\frac{1}{4},\frac{1}{4}}(t)=t+\frac{1}{3}[t(t-1)]^{\frac{1}{2}}.\]
On the other hand, in our original version of this paper, we also proved that among all the solutions $\lambda_{r,s}(t)$ to PVI$(\frac{9}{8},$ $\frac{-1}{8},\frac{1}{8},\frac{3}%
{8})$, $\lambda_{\frac{1}{3},0}(t)$ is the only solution which satisfies $\lambda(t)\notin\{0,1,t,\infty\}$ for all $t\in\mathbb{C}\setminus\{0,1\}$, which is a generalization of Theorem A-(i). After we communicated this statement to Eremenko, he and his coauthors \cite{Eremenko} found a simple proof of this statement and generalize it to the general parameters $(\alpha,\beta,\gamma,\delta)$. Therefore, we omit our proof of this statement in the current version. Together with their proof \cite{Eremenko}, it is known that $\lambda_{\frac{1}{3},0}(t)$ satisfies the following algebraic equation:
\[3\lambda^4-4t\lambda^3-4\lambda^3+6t\lambda^2-t^2=0.\]
\end{remark}

The paper is organized as follows: In \S 2, we introduce $Z_{r,s}^{(2)}(\tau)$
which is the denominator of (\ref{625}) and study its zeros. By connecting the
zeros of $Z_{r,s}^{(2)}(\tau)$ with the poles of $\lambda_{r,s}(t)$ (see
Theorem \ref{simple-zero}), we prove Theorem \ref{theorem1-7 copy(1)} in \S 2.
Next, we will count the number of the poles of an algebraic solution and the
explicit formulae are obtained in \S 3. Finally, by applying Theorem
\ref{thm number}, we prove Theorem \ref{main-thm copy(2)} in \S 4.

\section{Poles of solutions and zeros of premodular forms}

In this section, we are going to prove Theorem \ref{theorem1-7 copy(1)}.
Define $Z_{r,s}^{(2)}(\tau)$ to be the denominator of (\ref{625}), i.e.%
\begin{equation}
Z_{r,s}^{(2)}(\tau)\doteqdot Z_{r,s}(\tau)^{3}-3\wp(r+s\tau|\tau)Z_{r,s}%
(\tau)-\wp^{\prime}(r+s\tau|\tau). \label{II-5}%
\end{equation}
To study the poles of $\lambda_{r,s}(t)$, it is important to study the zeros
of $Z_{r,s}^{(2)}(\tau)$ for $\left(  r,s\right)  \in \mathbb{R}^{2}%
\backslash \frac{1}{2}\mathbb{Z}^{2}$. It was proved in \cite{CKLW} that%
\begin{equation}
Z_{r,s}(\tau)=\pm Z_{r^{\prime},s^{\prime}}(\tau)\Longleftrightarrow \left(
r,s\right)  \equiv \pm \left(  r^{\prime},s^{\prime}\right)  ,
\operatorname{mod}\mathbb{Z}^{2}, \label{Z-invariance}%
\end{equation}
which implies that%
\begin{equation}
Z_{r,s}^{(2)}(\tau)=\pm Z_{r^{\prime},s^{\prime}}^{(2)}(\tau
)\Longleftrightarrow \left(  r,s\right)  \equiv \pm \left(  r^{\prime},s^{\prime
}\right)  \operatorname{mod}\mathbb{Z}^{2}\text{.} \label{invariance}%
\end{equation}
From (\ref{Z-invariance})-(\ref{invariance}) and (\ref{625}), we proved in
\cite{Chen-Kuo-Lin2} that (the $\Longrightarrow$ part is not trivial)
\begin{equation}
\wp(p_{r,s}(\tau)|\tau)=\wp(p_{r^{\prime},s^{\prime}}(\tau)|\tau
)\Longleftrightarrow \left(  r,s\right)  \equiv \pm \left(  r^{\prime},s^{\prime
}\right)  \operatorname{mod}\mathbb{Z}^{2}\text{.} \label{invariant}%
\end{equation}
In particular, we have
\begin{equation}
\lambda_{r,s}(t)=\lambda_{r^{\prime},s^{\prime}}(t)\Longleftrightarrow
(r^{\prime},s^{\prime})\equiv \pm(r,s)\operatorname{mod}\mathbb{Z}^{2}\text{.}
\label{equal}%
\end{equation}
By (\ref{invariance})-(\ref{equal}), it is suitable to restrict $\left(
r,s\right)  $ on the set $[0,1]\times \lbrack0,\frac{1}{2}]\backslash \frac
{1}{2}\mathbb{Z}^{2}$. Define the four open triangles as follows:%
\[%
\begin{array}
[c]{l}%
\triangle_{0}:=\{(r,s)\mid0<r,s<\tfrac{1}{2},\text{ }r+s>\tfrac{1}{2}\},\\
\triangle_{1}:=\{(r,s)\mid \tfrac{1}{2}<r<1,\text{ }0<s<\tfrac{1}{2},\text{
}r+s>1\},\\
\triangle_{2}:=\{(r,s)\mid \tfrac{1}{2}<r<1,\text{ }0<s<\tfrac{1}{2},\text{
}r+s<1\},\\
\triangle_{3}:=\{(r,s)\mid r>0,\text{ }s>0,\text{ }r+s<\tfrac{1}{2}\}.
\end{array}
\]
Clearly $[0,1]\times \lbrack0,\frac{1}{2}]=\cup_{k=0}^{3}\overline
{\triangle_{k}}$. Remark that
\[
Z_{r,s}^{(2)}(\tau)\equiv0\text{ for }(r,s)\in \{(0,\tfrac{1}{2}),(\tfrac{1}%
{2},0),(\tfrac{1}{2},\tfrac{1}{2})\}+\mathbb{Z}^{2},
\]%
\[
Z_{r,s}^{(2)}(\tau)\equiv \infty \text{ for }(r,s)\in \mathbb{Z}^{2}.
\]
Define%
\begin{equation}
F_{0}:=\{ \tau \in \mathbb{H}\mid0\leq \operatorname{Re}\tau \leq1,\,|\tau
-\tfrac{1}{2}|\geq \tfrac{1}{2}\}. \label{g2}%
\end{equation}
It is known (cf. \cite{CKLW}) that $F_{0}$ is a fundamental domain for
$\Gamma_{0}\left(  2\right)  :=\{ \gamma=(a_{ij})\in SL(2,\mathbb{Z}%
)\,|\,a_{21}\equiv0 \operatorname{mod}2\}$. The following theorem gives us
information about the zeros of $Z_{r,s}^{(2)}(\tau)$ and plays an important
role to study the poles of $\lambda_{r,s}(t)$.\emph{\medskip}

\noindent \textbf{Theorem C.} \cite{Chen-Kuo-Lin3} \emph{Let }$(r,s)\in
\lbrack0,1]\times \lbrack0,\frac{1}{2}]\backslash \frac{1}{2}\mathbb{Z}^{2}%
$\emph{. Then }$Z_{r,s}^{(2)}(\tau)=0$\emph{ has a solution }$\tau$\emph{ in
}$F_{0}$\emph{ if and only if }$(r,s)\in \triangle_{1}\cup \triangle_{2}%
\cup \triangle_{3}$\emph{. } \emph{Furthermore, for any }$(r,s)\in \triangle
_{1}\cup \triangle_{2}\cup \triangle_{3}$\emph{, the solution }$\tau \in F_{0}%
$\emph{ is unique and satisfies }$\tau \in \mathring{F}_{0}$\emph{. In
particular, }$Z_{r,s}^{(2)}(\tau)\not =0$ \emph{for any }$\tau \in \partial
F_{0}$ \emph{and} $( r,s) \in \mathbb{R}^{2}\backslash \frac{1}{2}\mathbb{Z}%
^{2}$.\emph{\medskip}

We will use Theorem C to prove Theorem \ref{theorem1-7 copy(1)}. Our proof is
based on the following result to connect the poles of a solution
$\lambda_{r,s}(t)$ with zeros of $Z_{r,s}^{(2)}(\tau)$.

\begin{theorem}
\label{simple-zero}Fix any $(r,s)\in \mathbb{C}^{2}\backslash \frac{1}%
{2}\mathbb{Z}^{2}$ and $\tau_{0}\in \mathbb{H}$. Then $p_{r,s}(\tau_{0})=0$ in
$E_{\tau_{0}}$, or equivalently $t_{0}=t(\tau_{0})\not \in \{0,1,\infty \}$ is
a pole of $\lambda_{r,s}(t)$, if and only if either $r+s\tau_{0}\in
\Lambda_{\tau_{0}}$ or $r+s\tau_{0}\not \in \Lambda_{\tau_{0}}$ and
$Z_{r,s}^{(2)}(\tau_{0})=0$.
\end{theorem}

To prove Theorem \ref{simple-zero}, we have to prove that in the formula
(\ref{625}), the numerator and denominator can not vanish simultaneously.

\begin{lemma}
\label{lem-II-1}Fix any $(r,s)\in \mathbb{C}^{2}\backslash \frac{1}{2}%
\mathbb{Z}^{2}$ and $\tau \in \mathbb{H}$. Under the above notations, if
$Z_{r,s}^{(2)}(\tau)=0$,\ then%
\[
3\wp^{\prime}(\alpha)Z_{r,s}^{2}+\left(  12\wp^{2}(\alpha)-g_{2}\right)
Z_{r,s}+3\wp(\alpha)\wp^{\prime}(\alpha)\not =0.
\]

\end{lemma}

\begin{proof}
In the following, we denote $Z_{r,s}(\tau)$ simply by $Z$. First we recall the
well-known result%
\begin{equation}
g_{2}^{3}-27g_{3}^{2}=16(e_{1}-e_{2})^{2}(e_{2}-e_{3})^{2}(e_{3}-e_{1}%
)^{2}\not =0. \label{II-10}%
\end{equation}
Assume by contradiction that
\begin{equation}
Z_{r,s}^{(2)}(\tau)=Z^{3}-3\wp(\alpha)Z-\wp^{\prime}(\alpha)=0 \label{II-35}%
\end{equation}
and%
\begin{equation}
3\wp^{\prime}(\alpha)Z^{2}+\left(  12\wp^{2}(\alpha)-g_{2}\right)
Z+3\wp(\alpha)\wp^{\prime}(\alpha)=0. \label{II-6}%
\end{equation}

First we claim:%
\begin{equation}
Z(\tau)\not =0\text{ and }\wp^{\prime}(\alpha)\not =0\text{.} \label{II-8}%
\end{equation}
If $Z(\tau)=0$, then (\ref{II-35}) gives $\wp^{\prime}(\alpha)=0$, namely
\begin{equation}
\alpha=r+s\tau \in E_{\tau}[2]\setminus \Lambda_{\tau}\text{ and }Z(\tau)=0.
\label{II-7}%
\end{equation}
So we may assume $r+s\tau=\frac{\omega_{k}}{2}+m+n\tau$ for some $\left(
m,n\right)  \in \mathbb{Z}^{2}$ and $k\in \{1,2,3\}$. Consequently,
\[
\left(  r-m\right)  +\left(  s-n\right)  \tau-\frac{\omega_{k}}{2}=0,
\]
\[
\left(  r-m\right)  \eta_{1}+\left(  s-n\right)  \eta_{2}-\frac{\eta_{k}}%
{2}=0,
\]
which implies $(r,s)\in \frac{1}{2}\mathbb{Z}^{2}$ because the non-degeneracy
of $%
\begin{pmatrix}
1 & \tau \\
\eta_{1}(\tau) & \eta_{2}(\tau)
\end{pmatrix}
$, a contradiction. Similarly, if $\wp^{\prime}(\alpha)=0$, then $12\wp
^{2}(\alpha)-g_{2}=2\wp^{\prime \prime}(\alpha)\not =0$ and so (\ref{II-6})
gives $Z(\tau)=0$, again a contradiction. This proves the claim.

Now we apply the Euclidean algorithm for (\ref{II-35}) and (\ref{II-6}).
Multiplying (\ref{II-6}) by $Z$, (\ref{II-35}) by $3\wp^{\prime}(\alpha)$ and
adding them together, we obtain%
\begin{equation}
\left(  12\wp^{2}(\alpha)-g_{2}\right)  Z^{2}+12\wp(\alpha)\wp^{\prime}%
(\alpha)Z+3\wp^{\prime}(\alpha)^{2}=0. \label{II-34}%
\end{equation}
By (\ref{II-34}), (\ref{II-6}) and (\ref{II-33}), we can eliminate $Z^{2}$
term and obtain
\begin{equation}
\left[  12g_{2}\wp(\alpha)^{2}+36g_{3}\wp(\alpha)+g_{2}^{2}\right]
Z=-3\wp^{\prime}(\alpha)\left[  2g_{2}\wp(\alpha)+3g_{3}\right]  .
\label{II-9}%
\end{equation}
Consequently, multiplying (\ref{II-6}) by $(12g_{2}\wp(\alpha)^{2}+36g_{3}%
\wp(\alpha)+g_{2}^{2})^{2}$ and using (\ref{II-9}) lead to (write
$x=\wp(\alpha)$ for convenience)
\begin{align}
&  9(4x^{3}-g_{2}x-g_{3})(2g_{2}x+3g_{3})^{2}+x(12g_{2}x^{2}+36g_{3}%
x+g_{2}^{2})^{2}\label{II-11}\\
&  -(12x^{2}-g_{2})(2g_{2}x+3g_{3})(12g_{2}x^{2}+36g_{3}x+g_{2}^{2}%
)=0.\nonumber
\end{align}
A straightforward calculation shows that (\ref{II-11}) is exactly%
\[
0=3(g_{2}^{3}-27g_{3}^{2})(4x^{3}-g_{2}x-g_{3})=3\wp^{\prime}(\alpha
)^{2}(g_{2}^{3}-27g_{3}^{2}),
\]
which contradicts to (\ref{II-10}) and (\ref{II-8}).
\end{proof}

\begin{corollary}
\label{cor1}For any $\left(  r,s\right)  \in \mathbb{C}^{2}\backslash \frac
{1}{2}\mathbb{Z}^{2}$, any zero of $Z_{r,s}^{( 2) }(\tau)$ is simple.
\end{corollary}

\begin{proof}
The Painlev\'{e} property implies that $\wp(p_{r,s}(\tau)|\tau)$ is
meromorphic in $\mathbb{H}$. Therefore, this assertion follows from Lemma
\ref{lem-II-1} and the fact that any pole of any solution of PVI$(\frac{9}%
{8},\frac{-1}{8},\frac{1}{8},\frac{3}{8})$ must be simple (see e.g.
\cite[Proposition 1.4.1]{GP}).
\end{proof}

Now we could prove Theorem \ref{simple-zero} by Theorem B.

\begin{proof}
[Proof of Theorem \ref{simple-zero}]By the expression (\ref{625}) of
$\wp(p_{r,s}(\tau)|\tau)$, we see that $p_{r,s}(\tau_{0})$ $=0$ in
$E_{\tau_{0}}$ implies either $r+s\tau_{0}\in \Lambda_{\tau_{0}}$ or
$Z_{r,s}^{(2)}(\tau_{0})=0$.

So it suffices to prove the other direction. If $r+s\tau_{0}\in \Lambda
_{\tau_{0}}$, without loss of generality we may assume $\alpha(\tau
_{0})=r+s\tau_{0}=0$. By letting $\alpha=\alpha(\tau)\rightarrow$ $\alpha
(\tau_{0})=0$ as $\tau \rightarrow \tau_{0}$, we have%
\begin{equation}
\wp(\alpha)=\frac{1}{\alpha^{2}}+O(\alpha^{2}),\text{ }\wp^{\prime}%
(\alpha)=\frac{-2}{\alpha^{3}}+O(\alpha), \label{II-172-0}%
\end{equation}%
\begin{equation}
Z_{r,s}=\frac{1}{\alpha}\left(  1-c_{0}\alpha-c_{1}\alpha^{2}+O(\alpha
^{3})\right)  , \label{II-172}%
\end{equation}
where%
\begin{equation}
c_{0}\doteqdot r\eta_{1}(\tau_{0})+s\eta_{2}(\tau_{0})=-2\pi is\not =0
\label{II-176-0}%
\end{equation}
and
\[
c_{1}\doteqdot \frac{r}{s}\eta_{1}^{\prime}(\tau_{0})+\eta_{2}^{\prime}%
(\tau_{0}).
\]
Here $r=-s\tau_{0}$ and the Legendre relation $\tau \eta_{1}(\tau)-\eta
_{2}(\tau)=2\pi i$ are used in (\ref{II-176-0}). Then we deduce from
(\ref{II-5}) and (\ref{II-172-0})-(\ref{II-176-0}) that%
\begin{equation}
Z_{r,s}^{(2)}(\tau)=\frac{3c_{0}^{2}}{\alpha}-(c_{0}^{3}-6c_{0}c_{1}%
)+O(\alpha), \label{II-173}%
\end{equation}%
\begin{align*}
&  3\wp^{\prime}(\alpha)Z_{r,s}^{2}+\left(  12\wp^{2}(\alpha)-g_{2}\right)
Z_{r,s}+3\wp(\alpha)\wp^{\prime}(\alpha)\\
&  =-\frac{6c_{0}^{2}}{\alpha^{3}}-\frac{12c_{0}c_{1}}{\alpha^{2}}%
+O(\alpha^{-1}),
\end{align*}
and so (\ref{625}) gives%
\begin{equation}
\wp(p_{r,s}(\tau)|\tau)=-\frac{c_{0}}{3\alpha(\tau)}+O(1)\rightarrow
\infty \text{ as }\tau \rightarrow \tau_{0}, \label{II-174}%
\end{equation}
which implies $p_{r,s}(\tau_{0})=0$ in $E_{\tau_{0}}$, namely $t_{0}$ is a
pole of $\lambda_{r,s}(t)$ whenever $r+s\tau_{0}\in \Lambda_{\tau_{0}}$.

If $\alpha(\tau_{0})=r+s\tau_{0}\not \in \Lambda_{\tau_{0}}$ and
$Z_{r,s}^{(2)}(\tau_{0})=0$, then it follows from (\ref{625}) and Lemma
\ref{lem-II-1} that $\wp(p_{r,s}(\tau_{0})|\tau_{0})=\infty$.
\end{proof}

We need another lemma for the proof of Theorem \ref{theorem1-7 copy(1)}.

\begin{lemma}
\label{lem2}Let $\tau \in \mathbb{H}$, $\tau^{\prime}=\gamma \cdot \tau$ and
$(s^{\prime},r^{\prime})=(s,r)\cdot \gamma^{-1}$ for some $\gamma \in
SL(2,\mathbb{Z})$. Then $Z_{r,s}^{(2)}(\tau)=0$ if and only if $Z_{r^{\prime
},s^{\prime}}^{(2)}(\tau^{\prime})=0$. In particular, $Z_{r,s}^{(2)}%
(\tau)\not =0$ for any $( r,s) \in \mathbb{R}^{2}\backslash \frac{1}%
{2}\mathbb{Z}^{2}$ and $\tau \in SL(2,\mathbb{Z})\cdot i\mathbb{R}^{+}$.
\end{lemma}

\begin{proof}
Consider the pair $(z,\tau)\in \mathbb{C\times H}$ and $z=r+s\tau$. For any
$\gamma=%
\begin{pmatrix}
a & b\\
c & d
\end{pmatrix}
\in SL(2,\mathbb{Z})$, conventionally $\gamma$ can act on $\mathbb{C\times H}$
by
\[
\gamma(z,\tau)\doteqdot(\frac{z}{c\tau+d},\gamma \cdot \tau)=(\frac{z}{c\tau
+d},\frac{a\tau+b}{c\tau+d}).
\]
Then%
\begin{equation}
\frac{z}{c\tau+d}=\frac{r+s\tau}{c\tau+d}=r^{\prime}+s^{\prime}\tau^{\prime
},\text{ where }\tau^{\prime}=\gamma \cdot \tau \text{, }(s^{\prime},r^{\prime
})=(s,r)\cdot \gamma^{-1}. \label{II-31-0}%
\end{equation}
Using
\begin{equation}
\wp \left(  \left.  r^{\prime}+s^{\prime}\tau^{\prime}\right \vert \tau^{\prime
}\right)  =\left(  c\tau+d\right)  ^{2}\wp \left(  r+s\tau|\tau \right)  ,
\label{II-30}%
\end{equation}
we proved in \cite{CKLW} that%
\begin{equation}
Z_{r^{\prime},s^{\prime}}(\tau^{\prime})=(c\tau+d)Z_{r,s}(\tau). \label{II-31}%
\end{equation}
Together with (\ref{II-30}), (\ref{II-31}) and the fact that $g_{2}(\tau)$ is
a modular form of weight $4$, we easily derive from (\ref{II-5}) that%
\begin{equation}
Z_{r^{\prime},s^{\prime}}^{(2)}(\tau^{\prime})=(c\tau+d)^{3}Z_{r,s}^{(2)}%
(\tau) \label{II-32-0}%
\end{equation}
and%
\begin{equation}
\wp \left(  p_{r^{\prime},s^{\prime}}(\tau^{\prime})|\tau^{\prime}\right)
=(c\tau+d)^{2}\wp \left(  p_{r,s}(\tau)|\tau \right)  . \label{II-32}%
\end{equation}
In particular, $Z_{r,s}^{(2)}(\tau)=0$ if and only if $Z_{r^{\prime}%
,s^{\prime}}^{(2)}(\tau^{\prime})=0$.

Now fix any $( r,s) \in \mathbb{R}^{2}\backslash \frac{1}{2}\mathbb{Z}^{2}$.
Suppose $Z_{r,s}^{(2)}(\tau_{0})=0$ for some $\tau_{0}\in SL(2,\mathbb{Z}%
)\cdot i\mathbb{R}^{+}$, then $\tau_{0}=\gamma \cdot \tau$ for some $\tau \in
i\mathbb{R}^{+}$ and $\gamma \in SL(2,\mathbb{Z})$. Let $(s^{\prime},r^{\prime
})=(s,r)\cdot \gamma$. Then we have $(r^{\prime},s^{\prime})\in \mathbb{R}%
^{2}\backslash \frac{1}{2}\mathbb{Z}^{2}$ and $( c\tau+d) ^{3}Z_{r^{\prime
},s^{\prime}}^{(2)}(\tau)=Z_{r,s}^{(2)}(\tau_{0})=0$. But by Theorem C,
$Z_{r^{\prime},s^{\prime}}^{(2)}(\tau)\not =0$, a contradiction.
\end{proof}

\begin{proof}
[Proof of Theorem \ref{theorem1-7 copy(1)}]Let $\lambda \left(  t\right)  $ be
an unitary solution of PVI$(\frac{9}{8},\frac{-1}{8},\frac{1}{8},\frac{3}{8}%
)$, i.e. $\lambda(t) =\lambda_{r,s}(t) $ for some $( r,s) \in \mathbb{R}%
^{2}\backslash \frac{1}{2}\mathbb{Z}^{2}$. Suppose $t_{0}\in \mathbb{R}%
\backslash \{0,1\}$ is a pole of $\lambda(t)$. Recall $t(\tau)  $
in (\ref{tr}). A result in the theory of the modular form says that $t\left(
i\mathbb{R}^{+}\right)  $ $=\left(  0,1\right)  $, $t\left(  S\cdot
i\mathbb{R}^{+}\right)  =( 1,+\infty)$ and $t\left(  U\cdot i\mathbb{R}%
^{+}\right)  =( -\infty,0) $ where $S=%
\begin{pmatrix}
1 & 1\\
0 & 1
\end{pmatrix}
$ and $U=%
\begin{pmatrix}
1 & 0\\
2 & 1
\end{pmatrix}
$. See e.g. \cite{AK,Lang2}. Thus there exists $\tau_{0}\in SL(2,\mathbb{Z}%
)\cdot i\mathbb{R}^{+}$ such that $t_{0}=t(\tau_{0})$ and $p_{r,s}(\tau
_{0})=0$ in $E_{\tau_{0}}$. By Theorem \ref{simple-zero} and $r+s\tau
_{0}\not \in \Lambda_{\tau_{0}}$, we obtain $Z_{r,s}^{(2)}(\tau_{0})=0$, which yields a
contradiction with Lemma \ref{lem2}. Therefore, $\lambda(t)$ has no poles in
$\mathbb{R}\backslash \{0,1\}$.
\end{proof}

\section{Poles of Algebraic solutions}

In this section, we want to find the number of poles of algebraic solutions.
For PVI$(\frac{9}{8},\frac{-1}{8},\frac{1}{8},\frac{3}{8})$, it is well known
that a solution is algebraic if and only if its monodromy group is finite, and
the monodromy group of an algebraic solution is always the dihedral group
$D_{N}$ of order $2N$ for some $N\in \mathbb{N}_{\geq3}$. In this case, by
Theorem B, any branch of this solution must be one of $\lambda_{r,s}(t) $
where $\left(  r,s\right)  \in Q_{N}$, the set of N-torsion points. For the
classification and related subjects of algebraic solutions of Painlev\'{e} VI
equation, we refer to \cite{Boalch,Dubrovin-Mazzocco,Lisovyy-Tykhyy,Mazzocco}
and references therein.

To count the number of poles, we have to know how many branches of an
algebraic solution might have. A branch of a solution $\lambda(t)$ might be considered a single-valued meromorphic function (still
denoted by $\lambda(t) $) restricted on the simply connected domain
$\mathbb{C}\backslash(-\infty,1]$ or equivalently, the single-valued
meromorphic function $\wp( p( \tau) |\tau) $ restricted on a fundamental
domain of $\Gamma( 2) $ (because $t( \tau) $ is invariant under the action of
$\Gamma( 2) $). Thus, two branches $\lambda_{r,s}( t) $ and $\lambda
_{r^{\prime},s^{\prime}}( t) $, $t\in \mathbb{C}\backslash(-\infty,1]$, belong
to the same solution if $\lambda_{r^{\prime},s^{\prime}}( t) $ is the analytic
continuation of $\lambda_{r,s}( t) $ along a closed path cross the axis
$(-\infty,1]$. We note that for any algebraic solution, $\lambda_{r,s}( t) $
has no poles on $\mathbb{R}\backslash \{ 0,1 \} $ by Theorem
\ref{theorem1-7 copy(1)}. Hence whether a branch is considered as defined on
$\mathbb{C}\backslash((-\infty, 0]\cup[1,+\infty))$ or $\mathbb{C}\backslash(-\infty,1]$ does not
affect our calculation below.

\begin{remark}
We recall that $\Gamma(N)$ is the $N$-principal congruence subgroup of
$SL(2,\mathbb{Z})$, defined by%
\[
\Gamma(N):=\{ \gamma \in SL(2,\mathbb{Z})|\gamma \equiv I_{2}\operatorname{mod}%
N\}.
\]
It is known that $e_{k}(  \tau )  $, $k=1,2,3$, are modular forms of
weight 2 with respect to $\Gamma(2)$. We refer to \cite{Lang2} for the basic
theory of modular forms.
\end{remark}

\begin{proposition}
\label{Prop-II-2}$\lambda_{r,s}(t)  $ and $\lambda_{r^{\prime
},s^{\prime}}(t)$ belong to the same solution of PVI$(\frac
{9}{8}$, $\frac{-1}{8},\frac{1}{8},\frac{3}{8})$ if and only if $(s,r)\equiv
\pm(s^{\prime},r^{\prime})\cdot \gamma$ $\operatorname{mod}\mathbb{Z}^{2}$ by
some $\gamma \in \Gamma(2) $.
\end{proposition}

Using (\ref{625}), (\ref{II-32}), and (\ref{invariance})-(\ref{equal}), the
proof of Proposition \ref{Prop-II-2} is the same as \cite[Proposition
4.4]{CKLW}, where PVI$(\frac{1}{8},\frac{-1}{8},\frac{1}{8},\frac{3}{8})$ is
studied. The same result also holds for Picard solutions. See \cite[Theorem
1]{Mazzocco}. So we omit the details here.

\begin{lemma}
\label{4-torsion} (i) For any $N$-torsion points $(r,s),$ $\lambda_{r,s}( t) $
belongs to one of the three solutions $\lambda_{0,\frac{1}{N}}(t)$,
$\lambda_{\frac{1}{N},0}(t)$ and $\lambda_{\frac{1}{N},\frac{1}{N}}(t)$.

(ii) If $N$ is odd, then all the three in (i) belong to the same solution; and
if $N$ is even, then all the three in (i) represent 3 different solutions.

(iii)%
\begin{equation}
\lambda_{\frac{1}{N},0}(1-t)=1-\lambda_{0,\frac{1}{N}}(t),\text{
\  \  \  \ }\lambda_{\frac{1}{N},\frac{1}{N}}\left(  \frac{1}{t}\right)
=\frac{\lambda_{0,\frac{1}{N}}(t)}{t}. \label{4solution}%
\end{equation}

\end{lemma}

\begin{proof}
We divide the proof into three steps.\medskip

\textbf{Step 1.} Fix an $N$-torsion point $(r,s)=(\frac{k_{1}}{N},\frac{k_{2}%
}{N})$ with $0\leq k_{1},k_{2}\leq N-1$ and $\gcd(k_{1},k_{2},N)=1$. We show
that $\lambda_{r,s}(t)$ belongs to the same solution as one of $\{
\lambda_{0,\frac{1}{N}}(t),\lambda_{\frac{1}{N},0}(t),\lambda_{\frac{1}%
{N},\frac{1}{N}}(t)\}$. By Proposition \ref{Prop-II-2}, it suffices to prove
that for some $(r^{\prime},s^{\prime})\in \{(0,\frac{1}{N}),(\frac{1}%
{N},0),(\frac{1}{N},\frac{1}{N})\}$,
\begin{equation}
(s,r)\equiv(s^{\prime},r^{\prime})\cdot \gamma \text{ }\operatorname{mod}%
\mathbb{Z}^{2}\text{ \ by some }\gamma \in \Gamma(2) . \label{gamma}%
\end{equation}
Denote $L=\gcd(k_{1},k_{2})$ and $k_{j}=m_{j}L$. Then we have $\gcd(L,N)=1$
and $\gcd(m_{1},m_{2})=1$.

\textbf{Case 1.} both $m_{1}$ and $m_{2}$ are odd.

Then there exist $l_{1},l_{2}\in \mathbb{Z}$ such that $l_{1}m_{1}+l_{2}%
m_{2}=1$. Letting%
\begin{equation}
\gamma_{1}=%
\begin{pmatrix}
l_{1} & -l_{2}\\
m_{2}-l_{1} & m_{1}+l_{2}%
\end{pmatrix}
\in \Gamma(2)\text{ if }l_{1}\text{ odd,} \label{gamma-1}%
\end{equation}%
\begin{equation}
\gamma_{1}=%
\begin{pmatrix}
m_{2}+l_{1} & m_{1}-l_{2}\\
-l_{1} & l_{2}%
\end{pmatrix}
\in \Gamma(2)\text{ if }l_{1}\text{ even,} \label{gamma-2}%
\end{equation}
we have%
\begin{equation}
(s,r)=\left(  \frac{Lm_{2}}{N},\frac{Lm_{1}}{N}\right)  =\left(  \frac{L}%
{N},\frac{L}{N}\right)  \cdot \gamma_{1}. \label{ga}%
\end{equation}
If $L$ is odd and $N$ is odd, since $\gcd(L,2(L-N))=1$, there exists
$d_{1},d_{2}\in \mathbb{Z}$ such that $d_{1}L+2d_{2}(L-N)=1$. Let%
\begin{equation}
\gamma_{2}=%
\begin{pmatrix}
L & L-N\\
-2d_{2} & d_{1}%
\end{pmatrix}
\in \Gamma(2), \label{ga-2}%
\end{equation}
then $(\frac{L}{N},\frac{L}{N})\in(\frac{1}{N},0)\cdot \gamma_{2}%
+\mathbb{Z}^{2}$. Together with (\ref{ga}), we see that (\ref{gamma}) holds by
letting $(r^{\prime},s^{\prime})=(0,\frac{1}{N})$ and $\gamma=\gamma_{2}%
\gamma_{1}$. If $L$ is even and $N$ is odd, since $\gcd(2L,L-N)=1$, there
exist $\tilde{d}_{1},\tilde{d}_{2}\in \mathbb{Z}$ such that $2\tilde{d}%
_{1}L+\tilde{d}_{2}(L-N)=1$. Let%
\begin{equation}
\gamma_{2}=%
\begin{pmatrix}
\tilde{d}_{2} & -2\tilde{d}_{1}\\
L & L-N
\end{pmatrix}
\in \Gamma(2), \label{ga-3}%
\end{equation}
then $(\frac{L}{N},\frac{L}{N})\in(0,\frac{1}{N})\cdot \gamma_{2}%
+\mathbb{Z}^{2}$, which implies that (\ref{gamma}) holds by letting
$(r^{\prime},s^{\prime})=(\frac{1}{N},0)$ and $\gamma=\gamma_{2}\gamma_{1}$.
If $L$ is odd and $N$ is even, then similarly as (\ref{gamma-1}%
)-(\ref{gamma-2}), there exists%
\[
\gamma_{2}=%
\begin{pmatrix}
a & b\\
c & d
\end{pmatrix}
\in \Gamma(2)
\]
such that $a+c=L$ and $b+d=L-N$. Clearly $(\frac{L}{N},\frac{L}{N})\in
(\frac{1}{N},\frac{1}{N})\cdot \gamma_{2}+\mathbb{Z}^{2}$ and so (\ref{gamma})
holds by letting $(r^{\prime},s^{\prime})=(\frac{1}{N},\frac{1}{N})$ and
$\gamma=\gamma_{2}\gamma_{1}$.

\textbf{Case 2.} $m_{1}$ is even and $m_{2}$ is odd.

Similarly as (\ref{ga-2}) there exists%
\[
\gamma_{1}=%
\begin{pmatrix}
m_{2} & m_{1}\\
\ast & \ast
\end{pmatrix}
\in \Gamma(2).
\]
Then%
\begin{equation}
(s,r)=\left(  \frac{Lm_{2}}{N},\frac{Lm_{1}}{N}\right)  =\left(  \frac{L}%
{N},0\right)  \cdot \gamma_{1}. \label{ga-1}%
\end{equation}
If $L$ is odd and $N$ is even, there exists%
\[
\gamma_{2}=%
\begin{pmatrix}
L & N\\
\ast & \ast
\end{pmatrix}
\in \Gamma(2).
\]
Then $(\frac{L}{N},0)\in(\frac{1}{N},0)\cdot \gamma_{2}+\mathbb{Z}^{2}$, which
implies that (\ref{gamma}) holds by letting $(r^{\prime},s^{\prime}%
)=(0,\frac{1}{N})$ and $\gamma=\gamma_{2}\gamma_{1}$. If $L$ is even and $N$
is odd, similarly as (\ref{ga-3}) there exists%
\[
\gamma_{2}=%
\begin{pmatrix}
\ast & \ast \\
L & N
\end{pmatrix}
\in \Gamma(2).
\]
Then $(\frac{L}{N},0)\in(0,\frac{1}{N})\cdot \gamma_{2}+\mathbb{Z}^{2}$, which
implies that (\ref{gamma}) holds by letting $(r^{\prime},s^{\prime})=(\frac
{1}{N},0)$ and $\gamma=\gamma_{2}\gamma_{1}$. If $L$ is odd and $N$ is odd,
then similarly as (\ref{gamma-1})-(\ref{gamma-2}) there exists%
\[
\gamma_{2}=%
\begin{pmatrix}
a & b\\
c & d
\end{pmatrix}
\in \Gamma(2)
\]
such that $a+c=L$ and $b+d=N$. Then $(\frac{L}{N},0)\in(\frac{1}{N},\frac
{1}{N})\cdot \gamma_{2}+\mathbb{Z}^{2}$ and so (\ref{gamma}) holds by letting
$(r^{\prime},s^{\prime})=(\frac{1}{N},\frac{1}{N})$ and $\gamma=\gamma
_{2}\gamma_{1}$.

\textbf{Case 3.} $m_{2}$ is even and $m_{1}$ is odd.

The proof is similar to Case 2, so we omit the details. This completes the
proof of Step 1.\medskip

\textbf{Step 2. }Suppose $N=2m+1$ is odd. By choosing $\gamma=-%
\begin{pmatrix}
4m+1 & 2m\\
2 & 1
\end{pmatrix}
\in \Gamma(2)$, we have%
\[
\left(  \frac{1}{N},0\right)  \cdot \gamma \equiv \left(  \frac{1}{N},\frac{1}%
{N}\right)  \text{ mod }\mathbb{Z}^{2}.
\]
Similarly, by choosing $\tilde{\gamma}=%
\begin{pmatrix}
1 & 0\\
2m & 1
\end{pmatrix}
$, we see that
\[
\left(  \frac{1}{N},\frac{1}{N}\right)  \cdot \tilde{\gamma}\equiv \left(
0,\frac{1}{N}\right)  \text{ mod }\mathbb{Z}^{2}.
\]
Thus when $N$ is odd, all three in (i) belong to the same solution.

Now suppose $N=2m$ is even. For any $\gamma=%
\begin{pmatrix}
a & b\\
c & d
\end{pmatrix}
\in \Gamma(2)$, $c$ is even and $d$ is odd. We see that $\left(  0,\frac{1}%
{N}\right)  \cdot \gamma=\left(  \frac{c}{2m},\frac{d}{2m}\right)  $. Since $d$
is odd, $\frac{d}{2m}\not \equiv 0$ mod $\mathbb{Z}$ which implies $\left(
0,\frac{1}{N}\right)  \cdot \gamma \not \equiv \left(  \frac{1}{N},0\right)  $
mod $\mathbb{Z}^{2}$ for any $\gamma \in \Gamma(2)$. Similarly, any two of
$\left \{  \left(  0,\frac{1}{N}\right)  ,\left(  \frac{1}{N},0\right)
,\left(  \frac{1}{N},\frac{1}{N}\right)  \right \}  $ can not be connected by
$\Gamma(2)$ and mod $\mathbb{Z}^{2}$. Thus when $N$ is even, all the three in
(i) represent 3 different solutions.\medskip

\textbf{Step 3.} We prove (\ref{4solution}).

Let $\gamma=%
\begin{pmatrix}
1 & -1\\
0 & 1
\end{pmatrix}
$ and $\tau^{\prime}=\gamma \cdot \tau=\tau-1$. Since $(\frac{1}{N},\frac{1}%
{N})=(\frac{1}{N},0)\cdot \gamma^{-1}$, we see from (\ref{II-31-0}) and
(\ref{II-32}) that%
\[
\wp \left(  p_{\frac{1}{N},\frac{1}{N}}(\tau^{\prime})|\tau^{\prime}\right)
=\wp \left(  p_{0,\frac{1}{N}}(\tau)|\tau \right)  .
\]
This, together with $e_{1}(\tau^{\prime})=e_{1}(\tau)$, $e_{2}(\tau^{\prime
})=e_{3}(\tau)$, $e_{3}(\tau^{\prime})=e_{2}(\tau)$ and (\ref{tr}), implies
$t(\tau^{\prime})=1/t(\tau)$ and%
\begin{equation}
\lambda_{\frac{1}{N},\frac{1}{N}}(t(\tau^{\prime}))=\frac{\wp \left(
p_{0,\frac{1}{N}}(\tau)|\tau \right)  -e_{1}(\tau)}{e_{3}(\tau)-e_{1}(\tau
)}=\frac{\lambda_{0,\frac{1}{N}}(t(\tau))}{t(\tau)}. \label{r+s}%
\end{equation}
This proves the second formula in (\ref{4solution}). Similarly, by letting
$\gamma=%
\begin{pmatrix}
0 & -1\\
1 & 0
\end{pmatrix}
$, it is easy to prove the first formula in (\ref{4solution}).
\end{proof}

Define%
\begin{equation}
M_{N}(\tau):=\prod_{(r,s)\in Q_{N}}Z_{r,s}^{(2)}(\tau). \label{f-N}%
\end{equation}
By (\ref{II-32-0}), $M_{N}(\tau)$ is a \emph{modular form} with respect to
$SL(2,\mathbb{Z})$ of weight $3|Q_{N}|$, where $|Q_{N}|=\#Q_{N}$.

To obtain the number of zeros of $M_{N}(\tau)$, we recall the classical
formula for counting zeros of modular forms. See \cite{Serre} for the
proof.\medskip

\noindent \textbf{Theorem D.} \textit{Let }$f(\tau)$\textit{ be a nonzero
modular form with respect to }$SL(2,\mathbb{Z})$\textit{ of weight }%
$k$\textit{. Then}%
\begin{equation}
\sum_{\tau \in \mathbb{H}\backslash \{i,\rho \}}\nu_{\tau}(f)+\nu_{\infty
}(f)+\frac{1}{2}\nu_{i}(f)+\frac{1}{3}\nu_{\rho}(f)=\frac{k}{12},
\label{modular}%
\end{equation}
\textit{where }$\rho:=e^{\pi i/3}$\textit{, }$\nu_{\tau}(f)$ \textit{denotes
the zero order of} $f$ \textit{at} $\tau$ \textit{and the summation over
}$\tau$\textit{ is performed modulo }$SL(2,\mathbb{Z})$\textit{
equivalence.\medskip}

To prove Theorem \ref{thm number}, by Theorem D we have to calculate the
asymptotics of $Z_{r,s}^{(2)}(\tau)$ as $\tau \rightarrow\infty$. By using
the $q$-expansions of $\wp(z|\tau)$ (cf. \cite[p.46]{Lang}) and $Z_{r,s}(\tau)$
(cf. \cite[(5.3)]{CKLW}), the asymptotics of $Z_{r,s}^{(2)}(\tau)$ at $\tau=\infty$
can be calculated. Because the calculation is straightforward and is already
done in \cite{Dar,CLW2}, we state the statement and omit the calculation here.

\begin{lemma}
\label{infinity-behavior copy(1)} Let $(r,s)\in \lbrack0,1)\times \lbrack
0,1)\backslash \frac{1}{2}\mathbb{Z}^{2}$ and $q=e^{2\pi i\tau}$,
$0\leq\operatorname{Re}\tau \leq1$. Then as $\tau \rightarrow\infty$,%
\begin{equation}%
\begin{array}
[c]{l}%
Z_{r,s}^{(2)}(\tau)=4\pi^{3}is(1-s)(2s-1)+o(1)\text{ if }s\in \left(
0,\frac{1}{2}\right)  \cup \left(  \frac{1}{2},1\right)  ,\\
Z_{r,s}^{(2)}(\tau)=-48\pi^{3}\sin(2\pi r)q+O(q^{2})\text{ if }s=0,\\
Z_{r,s}^{(2)}(\tau)=-12\pi^{3}\sin(2\pi r)q^{1/2}+O(q)\text{ \ if \ }s=1/2.
\end{array}
\label{asp-2}%
\end{equation}

\end{lemma}

Let $\phi(N) $ be the Euler function defined in (\ref{Euler}) if $N$ is an
integer, and be zero if $N$ is not an integer. Now by applying Theorem D to
$M_{N}(\tau) $, we have the following theorem.

\begin{theorem}
\label{N-torsion}For $N\geq3$, the total number $P(N)$ of zeros (counting
multiplicity) of $M_{N}(\tau) $ in the fundamental domain $F$ of
$SL(2,\mathbb{Z})$ is given by
\begin{equation}
P(N) =\frac{|Q_{N}|}{4}-\left[  \phi(N)  +\phi \left(  \frac{N}%
{2}\right)  \right]  . \label{pn}%
\end{equation}

\end{theorem}

\begin{proof}
First, it is easy to see that $\frac{|Q_{N}|}{4}\in \mathbb{Z}_{>0}$ for any
$N\geq3$ because
\[
|Q_{N}|=N^{2}\prod_{p|N,\ p\text{ prime}}\left(  \frac{p^{2}-1}{p^{2}}\right)
.
\]
Since $i\in \partial F_{0}$, by Theorem C we see that $Z_{r,s}^{(
2)}(i)\neq0$ for any $\left(  r,s\right)
\in \mathbb{R}^{2}\backslash \frac{1}{2}\mathbb{Z}^{2}$. Thus
\begin{equation}
\nu_{i}(M_{N}(\tau))=0. \label{i}%
\end{equation}
By \cite[Corollary 3.2]{Chen-Kuo-Lin3} where $Z_{r,s}^{(2)}(\rho)\not =0$ for
any $(r,s)\in \mathbb{R}^{2}\backslash \frac{1}{2}\mathbb{Z}^{2}$ is proved, we
have
\begin{equation}
\nu_{\rho}(M_{N}(\tau))=0. \label{ro}%
\end{equation}
Then we deduce from (\ref{i}), (\ref{ro}) and (\ref{modular}) that%
\begin{equation}
P(N)=\sum_{\tau \in \mathbb{H}\backslash \{i,\rho \}}\nu_{\tau}(M_{N}(\tau
))=\frac{|Q_{N}|}{4}-v_{\infty}\left(  M_{N}(\tau)\right)  . \label{a}%
\end{equation}
From Lemma \ref{infinity-behavior copy(1)}, we see that%
\begin{align}
v_{\infty}\left(  M_{N}(\tau)\right)   &  =\# \left \{  1\leq k_{1}%
<N|\gcd \left(  N,k_{1}\right)  =1\right \} \label{b}\\
&  +\frac{1}{2}\# \left \{  1\leq k_{1}<N|\gcd \left(  N/2,k_{1}\right)
=1\right \} \nonumber \\
&  =\phi(  N)  +\phi \left(  \frac{N}{2}\right)  .\nonumber
\end{align}
Here $\phi \left(  \frac{N}{2}\right)  =0$ whenever $N$ is odd. Combining
(\ref{a}) and (\ref{b}), we have
\[
P(N)=\frac{|Q_{N}|}{4}-\left[  \phi(N)  +\phi \left(\frac{N}%
{2}\right)  \right]  .
\]
where the summation over $\tau$ is performed modulo $SL(2,\mathbb{Z})$
equivalence. This completes the proof.\medskip
\end{proof}

For $N=3,4$, it is easy to see $P( N) =0$.

\begin{corollary}
\label{coro} For $N=3,4$ and $\left(  r,s\right)  \in Q_{N}$, $Z_{r,s}%
^{\left(  2\right)  }\left(  \tau \right)  \not =0$ for any $\tau \in \mathbb{H}$.
\end{corollary}

Now, we are going to prove Theorem \ref{thm number}. Before the proof, we have
following three discussions: Let $F$ be a fundamental domain of $SL(
2,\mathbb{Z})  $ defined by
\[
F:=\left \{  \tau \in \mathbb{H}\,|\,0\leq \text{Re}\tau<1\text{, }\left \vert
\tau \right \vert \geq1,\left \vert \tau-1\right \vert >1\right \}  \cup \left \{
\rho=e^{\pi i/3}\right \}  \text{.}%
\]

(i) It is well-known (cf. \cite{AK}) that $\Gamma(2)$ is a normal subgroup of
$SL(2,\mathbb{Z})$ and $SL(2,\mathbb{Z})/\Gamma(2) $ $=$ $\{I,S,ST,S^{2}%
T,TS^{-1},STS^{-1}\}$ where $S=%
\begin{pmatrix}
1 & 1\\
0 & 1
\end{pmatrix}
$ and $T=%
\begin{pmatrix}
0 & -1\\
1 & 0
\end{pmatrix}
$. So a fundamental domain $F_{2}$ of $\Gamma(2) $ can be obtained by $F$:%
\begin{equation}
F_{2}=F\cup SF\cup STF\cup S^{2}TF\cup TS^{-1}F\cup STS^{-1}F. \label{f2}%
\end{equation}
By a straightforward computation, we have
\[
F_{2}=\{ \tau \in \mathbb{H}\mid0\leq \operatorname{Re}\tau<2,\,|\tau-\tfrac
{1}{2}|\geq \tfrac{1}{2}\text{, }\left \vert \tau-\tfrac{3}{2}\right \vert
>\tfrac{1}{2}\}.
\]
Note that for any $\tau$ $\in F_{2}$, $t(\tau)\in \mathbb{R}\backslash \left \{
0,1\right \}  $ if and only if $\tau$ $\in i\mathbb{R}^{+}$ $\cup \{ \tau
\in \mathbb{H}||\tau-\tfrac{1}{2}|=\tfrac{1}{2}\} \cup \{ \tau \in \mathbb{H}%
|\operatorname{Re}\tau=1\}$. More precisely, $t\left(  i\mathbb{R}^{+}\right)
$ $=\left(  0,1\right)  $, $t\left(  \{ \tau \in \mathbb{H}||\tau-\tfrac{1}%
{2}|=\tfrac{1}{2}\} \right)  $ $=\left(  -\infty,0\right)  $ and $t\left(
1+i\mathbb{R}^{+}\right)  $ $=\left(  1,+\infty \right)  $. See \cite{AK,Lang2}%
. Hence the image of $F_{2}\backslash(i\mathbb{R}^{+}\cup \{ \tau \in
\mathbb{H}||\tau-\tfrac{1}{2}|=\tfrac{1}{2}\})$ is $\mathbb{C}\backslash
(-\infty,1]$. By (\ref{f2}) and $M_{N}(\tau)\not =0$ at $\tau \in \{i, \rho \}$,
we have
\[
\# \{ \tau \in F_{2}\mid M_{N}(\tau)=0\}=6 P( N) ,
\]
where the RHS is counted by multiplicity.\medskip

(ii) If $\left(  r,s\right)  \in Q_{N}$, then $Q_{N}\ni$ $\left(  r^{\prime
},s^{\prime}\right)  \doteqdot( 1-r,1-s) $ mod $\mathbb{Z}^{2}$ and
$Z_{r^{\prime},s^{\prime}}^{( 2) }( \tau) $ $=-Z_{r,s}^{( 2) }( \tau) $. Hence
if $Z_{r,s}^{( 2) }( \tau_{0})=0$ then $Z_{r^{\prime},s^{\prime}}^{(2) }(
\tau_{0})=0$. In other words, the order of each zero of $M_{N}(\tau)$ is
\emph{even}. Remark that $\lambda_{r,s}( t) \equiv \lambda_{r^{\prime
},s^{\prime}}(t)$ by (\ref{equal}). \medskip

(iii) By Corollary \ref{cor1}, $Z_{r,s}^{(2) }( \tau) $ has only simple zeros.
For two $N$-torsion points $\left(  r,s\right)  $ $\not \equiv \pm \left(
r^{\prime},s^{\prime}\right)  \operatorname{mod} \mathbb{Z}^{2} $, $Z_{r,s}^{(
2) }( \tau) $ and $Z_{r^{\prime},s^{\prime}}^{( 2) }( \tau) $ might
simultaneously vanish at the same $\tau_{0}\in F_{2}$. But because
(\ref{equal}) gives $\lambda_{r,s}( t) \not \equiv \lambda_{r^{\prime
},s^{\prime}}( t)$, either they belongs to two different algebraic solutions
(this might happen if $N$ is even) or they are two different branches of the
same algebraic solution $\lambda(t)$ (this must happen if $N$ is odd). In the
later case, we count the number of poles of $\lambda(t) $ at $t_{0}=t(\tau
_{0}) $ as $2$ (as multiplicity).

\begin{proof}
[Proof of Theorem \ref{thm number}]If $N$ is odd, then for all $(r,s)\in
Q_{N}$, $\lambda_{r,s}( t) $ belong to the same one algebraic solution
$\lambda(t)$. By (i), (ii) and the simple zero property (iii), we obtain%
\[
\text{the total number of poles of }\lambda(t) =6 P( N) \times \frac{1}%
{2}=3P(N).
\]

If $N$ is even, then we have three different solutions, namely $\lambda
_{\frac{1}{N},0}(t) $, $\lambda_{0,\frac{1}{N}}(t) $, $\lambda_{\frac{1}%
{N},\frac{1}{N}}( t)$. By (\ref{4solution}), each of them have the same number
of branches and the same number of poles. Let $\lambda( t) $ be any one of
them. Then by (i), (ii) and (iii), we have%
\[
\text{the total number of poles of }\lambda(t) =6P( N) \times \frac{1}{2}%
\times \frac{1}{3}=P( N) .
\]
This completes the proof.
\end{proof}

\section{The proof of Theorem \ref{main-thm copy(2)}}

In this section, we want to prove Theorem \ref{main-thm copy(2)}, namely the
following result.

\begin{theorem}
\label{main-thm}Among all such solutions $\lambda_{r,s}(t)$ of PVI$(\frac
{9}{8},\frac{-1}{8},\frac{1}{8},\frac{3}{8})$, where $(r,s)\in \mathbb{C}%
^{2}\backslash \frac{1}{2}\mathbb{Z}^{2}$, there are exactly four solutions
which have no poles in $\mathbb{C}\backslash \{0,1\}$. They are precisely
$\lambda_{0,\frac{1}{3}}(t)$, $\lambda_{0,\frac{1}{4}}(t)$, $\lambda_{\frac
{1}{4},0}(t)$ and $\lambda_{\frac{1}{4},\frac{1}{4}}(t)$.
\end{theorem}

First we consider the case $(r,s)\not \in \mathbb{R}^{2}$. Since
$(r,s)\in \mathbb{C}^{2}\backslash \mathbb{R}^{2}$, there are infinitely many
$\tau_{0}\in \mathbb{H}$ such that $r+s\tau_{0}\in \Lambda_{\tau_{0}}$. The
following result is a direct consequence of Theorem \ref{simple-zero}.

\begin{lemma}
\label{complex-pair}Let $(r,s)\in \mathbb{C}^{2}\backslash \mathbb{R}^{2}$. Then
$\lambda_{r,s}(t)$ has infinitely many poles.
\end{lemma}

\begin{proof}
[Proof of Theorem \ref{main-thm}]Let $(r,s)\in \mathbb{C}^{2}\backslash \frac
{1}{2}\mathbb{Z}^{2}$. By Theorem \ref{simple-zero} and Corollary \ref{coro},
$\lambda_{r,s}(t)$ has no poles in $\mathbb{C}\setminus \{0,1\}$ for $(r,s)\in
Q_{3}\cup Q_{4}$. Therefore, we only need to prove that $\lambda_{r,s}(t)$ has
poles in $\mathbb{C}\setminus \{0,1\}$ whenever $(r,s)\not \in Q_{3}\cup Q_{4}%
$. By Lemma \ref{complex-pair} and (\ref{equal}), we only need to consider
$(r,s)\in \lbrack0,1]\times \lbrack0,\frac{1}{2}]\backslash \frac{1}{2}%
\mathbb{Z}^{2}$. Then by Theorem \ref{thm number}, we see that $P(N) >0$
except $N=3,4$. Together with Theorem C, it is enough for us to consider%
\[
(r,s)\in \triangle_{0}\cup \cup_{k=0}^{3}\partial \triangle_{k}\backslash
\mathbb{Q}^{2}.
\]
For $(r,s)\in \cup_{k=0}^{3}\partial \triangle_{k}\backslash \mathbb{Q}^{2}$, we
have $\{r,s,r+s\} \cap \mathbb{Q}\not =\emptyset$. Taking $\gamma=%
\begin{pmatrix}
3 & 2\\
4 & 3
\end{pmatrix}
\in \Gamma(2)$ and letting $(s^{\prime},r^{\prime})=(s,r)\cdot \gamma
=(4r+3s,3r+2s)$, we deduce from $(r,s)\not \in \mathbb{Q}^{2}$ that
$\pm(r^{\prime},s^{\prime})\not \in \cup_{k=0}^{3}\partial \triangle
_{k}+\mathbb{Z}^{2}$. Then by replacing $(r^{\prime},s^{\prime})$ by some
element in $\pm(r^{\prime},s^{\prime})+\mathbb{Z}^{2}$, we may assume
$(r^{\prime},s^{\prime})\in \cup_{k=0}^{3}\triangle_{k}$. By means of
Proposition \ref{Prop-II-2}, $\lambda_{r,s}(t)$ and $\lambda_{r^{\prime
},s^{\prime}}(t)$ belong to the same solution. Therefore, we conclude that to
prove Theorem \ref{main-thm}, we only need to prove that $\lambda_{r,s}(t)$
has poles in $\mathbb{C}\backslash \{0,1\}$ provided that
\begin{equation}
(r,s)\in \triangle_{0}\backslash \mathbb{Q}^{2}=\{(r,s)\mid0<r,s<\tfrac{1}%
{2},\text{ }r+s>\tfrac{1}{2}\} \backslash \mathbb{Q}^{2}. \label{tri-0}%
\end{equation}

Fix any $(r,s)\in \triangle_{0}\backslash \mathbb{Q}^{2}$. The same argument as
(\ref{r+s}) gives%
\begin{equation}
\lambda_{r+s,s}\left(  \frac{1}{t}\right)  =\frac{\lambda_{r,s}(t)}{t}.
\label{r+s-1}%
\end{equation}
Notice that if $r+2s\not =1$, then $(r+s,s)\in \triangle_{1}\cup \triangle_{2}$.
Applying Proposition \ref{N-torsion}, it follows that $\lambda_{r+s,s}\left(
t\right)  $ has poles in $\mathbb{C}\backslash \{0,1\}$ and so does
$\lambda_{r,s}(t)$.

So it suffices to consider $r=1-2s$, which implies $s\in(\frac{1}{4},\frac
{1}{2})\backslash \mathbb{Q}$. If $s\in(\frac{1}{4},\frac{3}{8})$, then
$(r^{\prime},s^{\prime})\doteqdot(2-4s,s)\in \triangle_{1}\cup \triangle_{2}$,
which implies that $\lambda_{r^{\prime},s^{\prime}}(t)$ has poles in
$\mathbb{C}\backslash \{0,1\}$. Since
\[
(s^{\prime},r^{\prime}-1)=(s,r-2s)=(s,r)\cdot%
\begin{pmatrix}
1 & -2\\
0 & 1
\end{pmatrix}
,
\]
by applying Proposition \ref{Prop-II-2} and (\ref{equal}), we see that
$\lambda_{r,s}\left(  t\right)  $ belong to the same solution with
$\lambda_{r^{\prime},s^{\prime}}(t)$ and so has poles in $\mathbb{C}%
\backslash \{0,1\}$.

So we may assume $r=1-2s$ and $s\in(\frac{3}{8},\frac{1}{2})\backslash
\mathbb{Q}$. Then there exists $m\in \mathbb{N}$ such that either%
\begin{equation}
2mr<s<(2m+1)r, \label{case1}%
\end{equation}
or%
\begin{equation}
(2m-1)r<s<2mr. \label{case2}%
\end{equation}
Let $\gamma=%
\begin{pmatrix}
1 & 0\\
-2m & 1
\end{pmatrix}
\in \Gamma(2)$ and $(s^{\prime},r^{\prime})=(s,r)\cdot \gamma=(s-2mr,r)$. If
(\ref{case1}) holds, then $(r^{\prime},s^{\prime})=(r,s-2mr)\in \triangle_{3}$.
If (\ref{case2}) holds, then $(1-r^{\prime},-s^{\prime})=(1-r,2mr-s)\in
\triangle_{2}$. In both cases, by means of Proposition \ref{N-torsion} and
(\ref{equal}), we see that $\lambda_{r^{\prime},s^{\prime}}(t)$ has poles in
$\mathbb{C}\backslash \{0,1\}$. Since Proposition \ref{Prop-II-2} says that
$\lambda_{r,s}(t)$ and $\lambda_{r^{\prime},s^{\prime}}(t)$ belong to the same
solution, we conclude that $\lambda_{r,s}( t) $ has poles in $\mathbb{C}%
\backslash \{0,1\}$. The proof is complete.
\end{proof}

\medskip
{\bf Acknowledgement} The authors wish to thank the anonymous referee very much for his/her careful reading, valuable comments and pointing out the references \cite{Brezhnev, Shimomura} to us.

\end{document}